\documentclass[10pt, reqno]{amsart}
\usepackage{amsmath, amsthm, amscd, amsfonts, amssymb, graphicx, color}
\usepackage[bookmarksnumbered, colorlinks, plainpages]{hyperref}
\hypersetup{colorlinks=true,linkcolor=red, anchorcolor=green, citecolor=cyan, urlcolor=red, filecolor=magenta, pdftoolbar=true}
\usepackage{mathrsfs}

\newtheorem{theorem}{Theorem}[section]
\newtheorem{lemma}[theorem]{Lemma}

\newtheorem{corollary}[theorem]{Corollary}
\theoremstyle{definition}
\newtheorem{definition}[theorem]{Definition}

\newtheorem{conjecture}[theorem]{Conjecture}

\theoremstyle{remark}
\newtheorem{remark}[theorem]{Remark}
\numberwithin{equation}{section}

\begin{document}
\setcounter{page}{1}

\title[geometry of $\mathbf{F}_1$ 
]{Geometry of $\mathbf{F}_1$ 
and Cuntz-Krieger algebras}

\author[Igor V. Nikolaev]
{Igor V. Nikolaev$^1$}

\address{$^{1}$ Department of Mathematics and Computer Science, St.~John's University, 8000 Utopia Parkway,  
New York,  NY 11439, United States.}
\email{\textcolor[rgb]{0.00,0.00,0.84}{igor.v.nikolaev@gmail.com}}

\dedicatory{All data are available as part of the manuscript}




\begin{abstract}
We study a natural map between projective varieties $V(\mathbf{F}_{1})$ over the field with one element 
and the Cuntz-Krieger algebras $\mathcal{O}_A$.   Using the $K$-theory of $\mathcal{O}_A$, we calculate the Frobenius 
action and cardinality of the set   $V(\mathbf{F}_{1^r})$.   It is proved that the zeta function of  $V(\mathbf{F}_{1})$ satisfies 
all Weil's Conjectures except for  an analog of the Riemann hypothesis.  We use the crossed product structure of $\mathcal{O}_A$ 
to establish a morphism of the schemes $\operatorname{Spec} ~(\mathbf{Z})\to \operatorname{Spec} ~(\mathbf{F}_{1})\simeq \{\operatorname{pt}\}$. 
 \end{abstract}

\maketitle

\section{Introduction}
The field with one element $\mathbf{F}_1$ is a powerful concept akin in spirit  to the ``non-existent'' square root of  $-1$. 
The history of $\mathbf{F}_1$ began  with the visionary work  [Tits 1957] \cite[Section 13]{Tit1} followed by the series of remarkable papers  [Manin 1995] \cite[Section 1.6]{Man1},
[Kapranov \& Smirnov 1995] \cite{KapSmi1}, [Soul\'e 2004] \cite{Sou1},  [Deitmar 2006] \cite{Dei1}, [Connes, Consani \& Marcolli 2009] \cite{ConConMar1}, [To\"en \& Vaqui\'e 2009]
\cite[Section 3.3]{ToeVaq1}, [Lorscheid 2018]  \cite{Lor1}  and others. 

The aim of our note is a natural map (functor) between projective varieties $V(\mathbf{F}_{1})$  and the Cuntz-Krieger algebras $\mathcal{O}_{A}$ 
(Theorem \ref{thm1.1}).  
It is proved, that many of the axioms of $V(\mathbf{F}_{1})$ conjectured in  \cite{KapSmi1}, \cite{Man1} and \cite{Sou1}
follow from the $K$-theory of the $C^*$-algebra $\mathcal{O}_{A}$ (Corollaries \ref{cor1.2} and \ref{cor1.3}). 
To formalize our results, we shall use  the following notation.

  The Cuntz-Krieger algebra is a  $C^*$-algebra $\mathcal{O}_A$
generated by the  partial isometries $s_1,\dots, s_m$ which satisfy  the relations:
\begin{equation}
\left\{
\begin{array}{ccc}
s_1^*s_1 &=& a_{11} s_1s_1^*+a_{12} s_2s_2^*+\dots+a_{1m}s_ms_m^*\\ 
s_2^*s_2 &=& a_{21} s_1s_1^*+a_{22} s_2s_2^*+\dots+a_{2m}s_ms_m^*\\ 
                  &\dots&\\
s_m^*s_m &=& a_{m1} s_1s_1^*+a_{m2} s_2s_2^*+\dots+a_{mm}s_ms_m^*,             
\end{array}
\right.
\end{equation}
where $A=(a_{ij})$ is a square matrix with  $a_{ij}\in \{0, 1, 2, \dots \}$ 
[Cuntz \& Krieger 1980] \cite{CunKri1}.
Let $q=p^r$ and $V(\mathbf{F}_q)$ be an $n$-dimensional projective variety over
the Galois field  $\mathbf{F}_q$. 
We shall use a natural map \cite{Nik1}:
\begin{equation}\label{eq1.2}
V(\mathbf{F}_q)\to \mathcal{O}_{\operatorname{Mk}_q}, 
\end{equation}
where the non-negative matrix $\operatorname{Mk}_q\in M_b(\mathbf{Z})$ corresponds to the Markov endomorphism
and $b$ is the largest Betti number of the projective variety $V(\mathbf{C})$;
we refer the reader to \cite[Section 6.7.2]{N} for the terminology and details. 
 It is known \cite[Theorem 1]{Nik1} that $|V(\mathbf{F}_q)|=\sum_{i=0}^{2n} (-1)^i ~\operatorname{tr} (\varepsilon_i^{\pi_i(q)})$,
 where $\varepsilon_i$ are algebraic units, $\pi_i(q)$ is an integer-valued function of $q$ and  $\operatorname{tr}$  is the trace
 of an algebraic integer; these are $K$-theory invariants of the Cuntz-Krieger algebra $\mathcal{O}_{\operatorname{Mk}_q}$
 \cite[Section 1]{Nik1}. 
 A critical observation is this. The correspondence (\ref{eq1.2}) extends seemlessly to the prime $p=1$. 
 Indeed, $\det \operatorname{Mk}_q=q^n$ and,  if $q=1^r=1$, one gets an invertible matrix   $A\in GL_b(\mathbf{Z})$;
 hence a well-defined map:
\begin{equation}\label{eq1.3}
V(\mathbf{F}_{1^r})\to \mathcal{O}_{A^r}.
\end{equation}
Likewise, the Betti numbers $b_i\le b$ give rise to the matrices $\{A_i\in GL_{b_i}(\mathbf{Z}) ~|~0\le i\le 2n\}$ \cite[Section 6.7.2]{N}.  
Our main result can be formulated as follows. 
\begin{theorem}\label{thm1.1}
The map $V(\mathbf{F}_{1^r})\to \mathcal{O}_{A^r}$ is a covariant 
functor, such that:

\medskip
(i) the action of Frobenius map $x\mapsto x^{1^r}$ on the $i$-th $\ell$-adic cohomology $H^i(V; \mathbf{Q}_{\ell})$ 
is given by the matrices  $\{A_i^r ~|~ 0\le i\le 2n\}$;

\smallskip
(ii) $|V(\mathbf{F}_{1^r})|=\sum_{i=0}^{2n} (-1)^i\operatorname{tr} (A_i^r)$.
\end{theorem}
\begin{remark}\label{rmk1.3}
Theorem \ref{thm1.1} fails for the commutative rings (Theorem \ref{thm4.1}). 
The obstacle is a problem of units (Section 4). This  problem 
has a solution in terms of  real multiplication  [Manin 2004] \cite{Man2};
hence Theorem \ref{thm1.1}.  
\end{remark}
\begin{corollary}\label{cor1.2}
 The zeta function  $\zeta_{V(\mathbf{F}_1)}(s):=\exp\left(\sum_{r=1}^{\infty} |V(\mathbf{F}_{1^r})|\frac{s^r}{r} \right)$
 of $V(\mathbf{F}_1)$   has the following properties:  

\medskip
(i)  $\zeta_{V(\mathbf{F}_1)}(s)=\frac{P_1(s)\dots P_{2n-1}(t)}{P_0(s)\dots P_{2n}(s)}$ is a rational function,
where  $P_i(s)\in \mathbf{Z}[s]$ and $P_0(s)=P_{2n}(s)=1-s$;

\smallskip
(ii) $\zeta_{V(\mathbf{F}_1)}(s)$ satisfies the functional equation  $\zeta_{V(\mathbf{F}_1)}\left(\frac{1}{s}\right)=\pm s^{\chi(V)} \zeta_{V(\mathbf{F}_1)}(s)$,
where $\chi(V)$ is the Euler-Poincar\'e characteristic of $V(\mathbf{C})$;

\smallskip
(iii)  $\deg~P_i(s)=b_i$, where $b_i$ is the $i$-th Betti number of $V(\mathbf{C})$.  
\end{corollary}
\begin{remark}\label{rmk1.4}
Corollary \ref{cor1.2} is an extention to the prime  $p=1$ of  the well-known  Weil's Conjectures,  except for an analog of the Riemann hypothesis;
see also Conjecture \ref{cnj4.3}.  
\end{remark}
Let  $\mathscr{K}$ be the 
$C^*$-algebra of compact operators on a Hilbert space and consider the tensor product 
 $\mathcal{O}_A\otimes\mathscr{K}$. 
This stabilized  Cuntz-Krieger algebra   is a crossed product $C^*$-algebra [Fillmore 1996] \cite[Chapter 8]{F}.  
Specifically,  one gets an isomorphism $\mathcal{O}_A\otimes\mathscr{K}\cong \mathbb{A}_V\rtimes_{\sigma_A}\mathbf{Z}$,
where  $\mathbb{A}_V$ is a stationary $AF$-algebra given by matrix $A$  \cite[Chapter 7]{B} 
and the crossed product $\rtimes$ is taken  by the shift automorphism $\sigma_A$ of $\mathbb{A}_V$ \cite[Exercise 10.11.9 (b)]{B}.
The   $\mathbb{A}_V$ itself is defined by the projective  variety $V(k)$ over a number field $k$ via its $K$-theory 
invariants $(\Lambda, [I], K)$ \cite{Nik2}; hence the notation.  
We use the embedding $\mathbb{A}_V\hookrightarrow  \mathbb{A}_V\rtimes_{\sigma_A}\mathbf{Z}\cong \mathcal{O}_A\otimes\mathscr{K}$
to define an injective homomorphism $\mathbb{A}_V\to \mathcal{O}_A\otimes\mathscr{K}$. 
The latter induces  a continuous  map
$\operatorname{Spec} (\mathbb{A}_{V})\to\operatorname{Spec} ~(\mathcal{O}_A\otimes\mathscr{K})$,
where $\operatorname{Spec} (\mathcal{A})$ is the space of unitary equivalence classes of irreducible 
$\ast$-repersentations of the $C^*$-algebra $\mathcal{A}$ equipped with the (quotient of) relative weak $\ast$-topology [Fillmore 1996] \cite[Section 6.2]{F}. 
Since $\operatorname{Spec} (\mathbb{A}_{V}) \simeq V(k)$ (Lemma \ref{lm3.7})  and $\operatorname{Spec} ~(\mathcal{O}_A\otimes\mathscr{K}) \simeq V(\mathbf{F}_1)$
(Lemma \ref{lm3.8}), one gets  a map $\operatorname{Spec} ~(\mathbf{Z})\to \operatorname{Spec} ~(\mathbf{F}_{1})\simeq \{\operatorname{pt}\}$
between the corresponding bases of the schemes. 
\begin{corollary}\label{cor1.3}
 The map $\operatorname{Spec} ~(\mathbf{Z})\to \operatorname{Spec} ~(\mathbf{F}_{1})\simeq \{\operatorname{pt}\}$
is a morphism of schemes. 
\end{corollary}

\medskip
The paper is organized as follows.  A brief review of the preliminary facts is 
given in Section 2. Theorem \ref{thm1.1},  Corollaries \ref{cor1.2} and \ref{cor1.3}
are proved in Section 3.  The case of elliptic curves with complex multiplication 
is considered in Section 4.

\section{Preliminaries}
We briefly review the Cuntz-Krieger and $AF$-algebras, $K$-theory and trace cohomology.
We refer the reader to [Blackadar 1986] \cite{B}, [Cuntz \& Krieger 1980] \cite{CunKri1},  [Fillmore 1996] \cite{F}, 
[Handelman 1981] \cite{Han2}, [Manin 2004] \cite{Man2}  and \cite{N}  for a systematic account. 

\subsection{C*-algebras}
The $C^*$-algebra is an algebra  $\mathscr{A}$ over $\mathbf{C}$ with a norm 
$a\mapsto ||a||$ and an involution $\{a\mapsto a^* ~|~ a\in \mathscr{A}\}$  such that $\mathscr{A}$ is
complete with  respect to the norm, and such that $||ab||\le ||a||~||b||$ and $||a^*a||=||a||^2$ for every  $a,b\in \mathscr{A}$.  
Each commutative $C^*$-algebra is  isomorphic
to the algebra $C_0(X)$ of continuous complex-valued
functions on some locally compact Hausdorff space $X$. 
Any other  algebra $\mathscr{A}$ can be thought of as  a noncommutative  
topological space. 

\subsubsection{$AF$-algebras}
An  approximately finite-dimensional $C^*$-algebra ($AF$-algebra) $\mathbb{A}$ is defined to
be the  norm closure of an ascending sequence of   finite dimensional
$C^*$-algebras $M_n$,  where  $M_n$ is the $C^*$-algebra of the $n\times n$ matrices
with entries in $\mathbf{C}$. Here the index $n=(n_1,\dots,n_k)$ represents
the  semi-simple matrix algebra $M_n=M_{n_1}\oplus\dots\oplus M_{n_k}$.
The ascending sequence mentioned above  can be written as 
$M_1\buildrel\rm\varphi_1\over\longrightarrow M_2
   \buildrel\rm\varphi_2\over\longrightarrow\dots$,
where $M_i$ are the finite dimensional $C^*$-algebras and
$\varphi_i$ the homomorphisms between such algebras.  
If $\varphi_i=Const$, then the AF-algebra $\mathbb{A}$ is called 
 stationary. 
The homomorphisms $\varphi_i$ can be arranged into  a graph as follows. 
Let  $M_i=M_{i_1}\oplus\dots\oplus M_{i_k}$ and 
$M_{i'}=M_{i_1'}\oplus\dots\oplus M_{i_k'}$ be 
the semi-simple $C^*$-algebras and $\varphi_i: M_i\to M_{i'}$ the  homomorphism. 
One has  two sets of vertices $V_{i_1},\dots, V_{i_k}$ and $V_{i_1'},\dots, V_{i_k'}$
joined by  $a_{rs}$ edges  whenever the summand $M_{i_r}$ is contained in  $a_{rs}$
copies of the summand $M_{i_s'}$ under the embedding $\varphi_i$. 
As $i$ varies, one obtains an infinite graph called the   Bratteli diagram of the
AF-algebra.  The matrix $A=(a_{rs})$ is known as  a  partial multiplicity matrix;
an infinite sequence of $A_i$ defines a unique AF-algebra.
If   $\mathbb{A}$ is a stationary AF-algebra, then   $A_i=Const$
for all $i\ge 1$.

\subsubsection{Cuntz-Krieger algebras}
  The Cuntz-Krieger algebra is a  $C^*$-algebra $\mathcal{O}_A$
generated by the  partial isometries $s_1,\dots, s_m$ which satisfy  the relations:
\begin{equation}
\left\{
\begin{array}{ccc}
s_1^*s_1 &=& a_{11} s_1s_1^*+a_{12} s_2s_2^*+\dots+a_{1m}s_ms_m^*\\ 
s_2^*s_2 &=& a_{21} s_1s_1^*+a_{22} s_2s_2^*+\dots+a_{2m}s_ms_m^*\\ 
                  &\dots&\\
s_m^*s_m &=& a_{m1} s_1s_1^*+a_{m2} s_2s_2^*+\dots+a_{mm}s_ms_m^*,             
\end{array}
\right.
\end{equation}
where $A=(a_{ij})$ is a square matrix with  $a_{ij}\in \{0, 1, 2, \dots \}$ 
[Cuntz \& Krieger 1980] \cite{CunKri1}.
If  $\mathbb{A}$ be a stationary  $AF$-algebra given by matrix $A$
and $\mathscr{K}$ is the $C^*$-algebra of compact operators, 
then
\begin{equation}
\mathcal{O}_A\otimes \mathscr{K}\cong \mathbb{A}\rtimes_{\sigma_A}\mathbf{Z}, 
\end{equation}
where the crossed product is taken by the shift automorphism $\sigma_A$
of the Bratteli diagram of $\mathbb{A}$  \cite[Exercise 10.11.9 (b)]{B}. 
The $AF$-algebra $\mathbb{A}$ is a subalgebra of the stabilized Cuntz-Krieger algebra  $\mathcal{O}_A\otimes\mathscr{K}$,
i.e. there exists an injective homomorphism (embedding)  $\mathbb{A}\to\mathcal{O}_A\otimes\mathscr{K}$
[Cuntz \& Krieger 1980] \cite{CunKri1}.

\subsection{K-theory}
By $M_{\infty}(\mathscr{A})$ 
one understands the algebraic direct limit of the $C^*$-algebras 
$M_n(\mathscr{A})$ under the embeddings $a\mapsto ~\mathbf{diag} (a,0)$. 
The direct limit $M_{\infty}(\mathscr{A})$  can be thought of as the $C^*$-algebra 
of infinite-dimensional matrices whose entries are all zero except for a finite number of the
non-zero entries taken from the $C^*$-algebra $\mathscr{A}$.
Two projections $p,q\in M_{\infty}(\mathscr{A})$ are equivalent, if there exists 
an element $v\in M_{\infty}(\mathscr{A})$,  such that $p=v^*v$ and $q=vv^*$. 
The equivalence class of projection $p$ is denoted by $[p]$.   
We write $V(\mathscr{A})$ to denote all equivalence classes of 
projections in the $C^*$-algebra $M_{\infty}(\mathscr{A})$, i.e.
$V(\mathscr{A}):=\{[p] ~:~ p=p^*=p^2\in M_{\infty}(\mathscr{A})\}$. 
The set $V(\mathscr{A})$ has the natural structure of an abelian 
semi-group with the addition operation defined by the formula 
$[p]+[q]:=\mathbf{diag}(p,q)=[p'\oplus q']$, where $p'\sim p, ~q'\sim q$ 
and $p'\perp q'$.  The identity of the semi-group $V(\mathscr{A})$ 
is given by $[0]$, where $0$ is the zero projection. 
By the $K_0$-group $K_0(\mathscr{A})$ of the unital $C^*$-algebra $\mathscr{A}$
one understands the Grothendieck group of the abelian semi-group
$V(\mathscr{A})$, i.e. a completion of $V(\mathscr{A})$ by the formal elements
$[p]-[q]$.  The image of $V(\mathscr{A})$ in  $K_0(\mathscr{A})$ 
is a positive cone $K_0^+(\mathscr{A})$ defining  the order structure $\le$  on the  
abelian group  $K_0(\mathscr{A})$. The pair   $\left(K_0(\mathscr{A}),  K_0^+(\mathscr{A})\right)$
is known as a dimension group of the $C^*$-algebra $\mathscr{A}$. 
The scale $\Sigma(\mathscr{A})$ is the image in $K_0^+(\mathscr{A})$
of the equivalence classes of projections in the $C^*$-algebra $\mathscr{A}$. 
The $\Sigma(\mathscr{A})$ is a generating, hereditary and directed subset 
of  $K_0^+(\mathscr{A})$, i.e. (i) for each $a\in K_0^+(\mathscr{A})$ 
there exist $a_1,\dots, a_r\in\Sigma(\mathscr{A})$ such that 
$a=a_1+\dots+a_r$; (ii) if $0\le a\le b\in \Sigma(\mathscr{A})$, then $a\in\Sigma(\mathscr{A})$
and (iii) given $a,b\in\Sigma(\mathscr{A})$ there exists $c\in\Sigma(\mathscr{A})$,
such that $a,b\le c$.   Each  scale  can always be written as 
$\Sigma(\mathscr{A})=\{a\in K_0^+(\mathscr{A}) ~|~0\le a\le u\}$,
where $u$ is an  order unit of  $K_0^+(\mathscr{A})$.  
The pair  $\left(K_0(\mathscr{A}),  K_0^+(\mathscr{A})\right)$ and the
triple  $\left(K_0(\mathscr{A}),  K_0^+(\mathscr{A}), \Sigma(\mathscr{A})\right)$
are invariants of the Morita equivalence and isomorphism class of the 
$C^*$-algebra $\mathscr{A}$, respectively. 
If  $\mathbb{A}$ is an AF-algebra, then its scaled dimension group 
(dimension group, resp.) is a complete invariant of the isomorphism 
(Morita equivalence, resp.) class of $\mathbb{A}$, see e.g. \cite[Theorem 3.5.2]{N}. 
Let $\tau$ be the canonical trace on the AF-algebra  $\mathbb{A}$. 
Such a trace induces a homomorphism $\tau_*: K_0(\mathbb{A})\to\mathbf{R}$
and  we let $\mathfrak{m}:=\tau_*( K_0(\mathbb{A}))\subset\mathbf{R}$.  
If    $\mathbb{A}$   is the stationary AF-algebra given by a matrix $A\in GL(n, \mathbf{Z})$, 
then $\mathfrak{m}$ is a $\mathbf{Z}$-module 
in the number field $K=\mathbf{Q}(\lambda_A)$ generated by the Perron-Frobenius 
eigenvalue $\lambda_A$  of  the matrix $A$.    The endomorphism ring of $\mathfrak{m}$
is denoted by $\Lambda$ and the ideal class of $\mathfrak{m}$ is denoted by $[\mathfrak{m}]$. 
The triple  $(\Lambda, [\mathfrak{m}], K)$
is an invariant of the Morita equivalence class of $\mathbb{A}$  [Handelman 1981] \cite{Han2}.

\subsection{Trace cohomology}
 Let $V(k)$ be a  projective variety defined over an algebraic number field $k\subset \mathbf{C}$.
 Suppose that   $V(\mathbb{F}_q)$ is the reduction of $V(k)$ modulo  the prime ideal 
 $\mathfrak{P}\subset k$ corresponding to  $q=p^r$.   Denote by $\mathcal{A}_V$ the  Serre $C^*$-algebra 
of projective variety $V(K)$ \cite[Section 5.3.1]{N}.  
Consider the  stable $C^*$-algebra of $\mathcal{A}_V$,  i.e. the $C^*$-algebra  $\mathcal{A}_V\otimes \mathscr{K}$,
 where $\mathscr{K}$ is the $C^*$-algebra of compact operators. 
 Let $\tau: \mathcal{A}_V\otimes \mathscr{K}\to \mathbf{R}$   be the unique normalized trace (tracial state) on  $\mathcal{A}_V\otimes \mathscr{K}$, 
  i.e. a positive linear functional   of norm $1$  such that $\tau(yx)=\tau(xy)$ for all $x,y\in \mathcal{A}_V\otimes \mathscr{K}$
 [Blackadar 1986] \cite[p. 31]{B}.   
Recall that $\mathcal{A}_V$ is the crossed product $C^*$-algebra of the form
$\mathcal{A}_V\cong C(V)\rtimes \mathbf{Z}$,  where $C(V)$ is the 
commutative $C^*$-algebra of complex valued functions on $V$.  
 By the Pimsner-Voiculescu six term exact sequence for
$0\to K_0(C(V))\buildrel  i_*\over\to  K_0(\mathcal{A}_V)\to K_1(C(V))\to 0$, 
where   map  $i_*$  is induced by an  embedding of $C(V)$ 
into $\mathcal{A}_V$  [Blackadar 1986]  \cite[p. 83]{B}.  
 We  have $K_0(C(V))\cong K^0(V)$ and 
$K_1(C(V))\cong K^{-1}(V)$,  where $K^0$ and $K^{-1}$  are  the topological
$K$-groups of variety $V$ [Blackadar 1986]  \cite[p. 80]{B}. 
By  the Chern character formula,  one gets
$K^0(V)\otimes \mathbf{Q} \cong H^{even}(V; \mathbf{Q})$
and 
$K^{-1}(V)\otimes \mathbf{Q} \cong   H^{odd}(V; \mathbf{Q})$,
where $H^{even}$  ($H^{odd}$)  is the direct sum of even (odd, resp.) 
cohomology groups of $V$. 
 It is known,  that  $K_0(\mathcal{A}_V\otimes \mathscr{K})\cong K_0(\mathcal{A}_V)$  by
 stability of the $K_0$-group [Blackadar 1986]  \cite[p. 32]{B}.
Thus one gets the  commutative diagram shown in Figure 1, 
where $\tau_*$ denotes  a homomorphism  induced on $K_0$ by  the canonical  trace 
$\tau$ on the $C^*$-algebra  $\mathcal{A}_V\otimes \mathscr{K}$. 
Since $H^{even}(V):=\oplus_{i=0}^n H^{2i}(V)$ and  
$H^{odd}(V):=\oplus_{i=1}^n H^{2i-1}(V)$,   one gets  for each  $0\le i\le 2n$ 
 an injective  homomorphism 
 $H^i(V)\to  \mathbf{R}$.
 \begin{definition}
 An  additive abelian subgroup $\Lambda_i$  of real numbers 
 defined by the homomorphism  $H^i(V)\to  \mathbf{R}$
 is called $i$-th trace cohomology of $V$. 
  \end{definition}
 The $\Lambda_i$ is   a  pseudo-lattice  
 [Manin 2004]  \cite[Section 1]{Man2}. 
Recall that  endomorphisms  of a pseudo-lattice are given as 
multiplication of points of $\Lambda_i$ by the real numbers $\alpha$
such that $\alpha\Lambda_i\subseteq\Lambda_i$.   It is known that
$End~(\Lambda_i)\cong \mathbf{Z}$ or $End~(\Lambda_i)\otimes \mathbf{Q}$
is a real algebraic number field such that $\Lambda_i\subset  End~(\Lambda_i)\otimes \mathbf{Q}$,
see e.g.  [Manin 2004]  \cite[Lemma 1.1.1]{Man1}  for the case of quadratic fields. 
We shall write $\varepsilon_i$ to denote the unit of the order in the field $K_i:=End~(\Lambda_i)\otimes \mathbf{Q}$,  
which induces the  shift automorphism  of $\Lambda_i$. 
  Let $p$ be a  good prime  in  the reduction $V(\mathbb{F}_q)$ of 
  complex projective variety $V(K)$ modulo a prime ideal over  $q=p^r$.  
   Consider a sub-lattice $\Lambda_i^{q}$ of $\Lambda_i$ of the index $q$; 
  by an  index of  the sub-lattice we understand  its  index as an abelian subgroup of $\Lambda_i$.
We shall write  $\pi_i(q)$ to  denote  an
 integer,   such that  multiplication by $\varepsilon_i^{\pi_i(q)}$ 
 induces  the shift automorphism of   $\Lambda_i^q$. 
  The trace of an algebraic number will be written as $\operatorname{tr}$.    
  The following result relates invariants $\varepsilon_i$ and $\pi_i(q)$ of the $C^*$-algebra
 $\mathcal{A}_V$ to the cardinality of the set $V(\mathbb{F}_q)$.   
\begin{theorem}\label{thm2.1} {\bf (\cite{Nik1})}
\quad $|V(\mathbb{F}_q)|=\sum_{i=0}^{2n}(-1)^i  ~\operatorname{tr}~\left(\varepsilon_i^{\pi_i(q)}\right)$.  
\end{theorem}

\bigskip
\begin{figure}
\begin{picture}(300,100)(0,5)
\put(160,72){\vector(0,-1){35}}
\put(80,65){\vector(2,-1){45}}
\put(240,65){\vector(-2,-1){45}}
\put(10,80){$ H^{even}(V)\otimes \mathbf{Q} 
\buildrel  i_*\over\longrightarrow  K_0(\mathcal{A}_V\otimes\mathscr{K})\otimes \mathbf{Q} 
\longrightarrow H^{odd}(V)\otimes \mathbf{Q}$}
\put(167,55){$\tau_*$}
\put(157,20){$\mathbf{R}$}
\end{picture}
\caption{Trace cohomology.}
\end{figure}

\section{Proof}
\subsection{Proof of Theorem \ref{thm1.1}}
For the sake of clarity, let us outline the main ideas. 
Notice that Weil's formula  $|V(\mathbf{F}_q)|=\sum_{i=0}^{2n} (-1)^i\operatorname{tr} ~(\operatorname{Fr}_q^i)$
is not valid, if  $q=1^r$,  due to a problem of units; see Section 4 for a detailed non-example. 
Specifically, consider  an $n$-dimensional abelian variety $A_{CM}$ with complex multiplication by a field $k_{CM}$.  
Then  $|A_{CM}(\mathbf{F}_q)|=\sum_{i=0}^{2n} (-1)^i\operatorname{tr} ~(\omega^i)$,
where $\omega^i\in k_{CM}$ are algebraic integers of the norm $q^n$. 
If  $q=1^r=1$, then $\omega^i$ have the norm $1$, i.e. are the units of $k_{CM}$.
 But the number of such (non-real) units and their traces in the CM-fields is always finite,
i.e. $\omega^i$ cannot  capture the  arithmetic of  $A_{CM}(\mathbf{F}_{1^r})$.  (Obviously,  an infinite number
 of the non-real units are needed for this purpose.)
The real multiplication (RM) [Manin 2003] \cite{Man2} avoids this problem, since the RM-fields
have a plethora of units;  hence an analog of Weil's formula 
$|V(\mathbb{F}_q)|=\sum_{i=0}^{2n}(-1)^i  ~\operatorname{tr}~\left(\varepsilon_i^{\pi_i(q)}\right)$
  is working well for the case  $q=1^r=1$. 
   Let us pass to a detailed argument using the trace cohomology (Section 2.3).  
\begin{lemma}\label{lm3.1}
The map $V(\mathbf{F}_{1^r})\to \mathcal{O}_{A^r}$ is a covariant 
functor. 
\end{lemma} 
\begin{proof}
(i)
Let $q=p^r$ and let $V(\mathbf{F}_q)$ be an $n$-dimensional projective variety 
over the Galois field $\mathbf{F}_q$. 
Denote by $\{b_i ~|~ 0\le i\le 2n\}$ the $i$-th Betti number of complex variety
$V(\mathbf{C})$.  If $\Lambda_i$ is the $i$-th trace cohomology of $V$ (Section 2.3),
one gets inclusions $\Lambda_i\subseteq\Lambda$, where $\Lambda$ is the pseudo-lattice corresponding 
to the maximal Betti number $b$  of $V$. 

\medskip
(ii) Let $(\Lambda, [\mathfrak{m}], K)$ be the Handelman triple defined by the pseudo-lattice $\Lambda$
(Section 2.2).  Suppose that  the corresponding stationary $AF$-algebra $\mathbb{A}$ is given by  a matrix $A\in GL_b(\mathbf{Z})$. 

\medskip
(iii) Recall \cite[Section 6.7.2]{N} that the Markov endomorphism $\operatorname{Mk}_q$ is acting on $\Lambda$
by  multiplication $q^n\omega$, where $\omega\in K$ is an algebraic unit such that  $\operatorname{tr} ~(\omega)=\operatorname{tr}~(A)$.
In particular, such an endomorphism is defined by a non-negative matrix $\operatorname{Mk}_q\in M_b(\mathbf{Z})$ such that   $\det  (\operatorname{Mk}_q)=q^n$.
We let $\mathcal{O}_{\operatorname{Mk}_q}$ be the corresponding Cuntz-Krieger algebra.

\medskip
(iv) On the other hand, it is known that the map $V(\mathbf{F}_q)\to \mathcal{O}_{\operatorname{Mk}_q}$ is a covariant functor 
between the category of projective varieties over finite fields and such of the Cuntz-Krieger algebras which preserves all natural morphisms 
of the respective categories (Theorem \ref{thm2.1}). 

\medskip
(v)  We shall assume $p=1$ and repeat the argument of items (i)-(iv). 
Since $q=1^r$, one gets from item (iii) that $\det  (\operatorname{Mk}_q)=1^{rn}=1$.
Notice that the entries of eigenvectors of the matrices $\operatorname{Mk}_q$ and $ A\in GL_b(\mathbf{Z})$
belong to the same pseudo-lattice  $\Lambda$ and thus   $\operatorname{Mk}_q$ must be a power $\pi(q)$ of $A$
(Section 2.3).  It is not hard to see, that $\pi (q)=r$.  Indeed,  $A$ is a matrix form of the fundamental unit $\omega\in K$,  see item (iii). 
On the other hand, $|V(\mathbf{F}_{1^r})|$ must be a monotone integer-valued function of $r$, such that $\pi(1^r)<\pi(1^{r+1})$ for 
all $r\ge 1$ (Theorem \ref{thm2.1}).
But the latter condition is satsfied if and only if $\pi(q)=\pi(1^r)=r$. 
 Thus   one gets $\operatorname{Mk}_{1^r}\cong A^r$. 
 
 \medskip
(vi) It remains to compare a covariant functor   $V(\mathbf{F}_q)\to \mathcal{O}_{\operatorname{Mk}_q}$ 
with the results of item (v).  We conclude therefore that the map  $V(\mathbf{F}_{1^r})\to \mathcal{O}_{A^r}$ is a covariant 
functor between  the category of projective varieties over $\mathbf{F}_1$ and such of the Cuntz-Krieger algebras.

\bigskip
Lemma \ref{lm3.1} is proved.
\end{proof}

\begin{lemma}\label{lm3.2}
The action of Frobenius map $x\mapsto x^{1^r}$ on the $i$-th $\ell$-adic cohomology $H^i(V; \mathbf{Q}_{\ell})$ 
is given by the matrices  $\{A_i^r ~|~ 0\le i\le 2n\}$. 
\end{lemma} 
\begin{proof}
(i)  Denote by $Fr_q^i$ a linear transformation of the $\ell$-adic cohomology $H^i(V; \mathbf{Q}_{\ell})$ 
induced by the Frobenius map acting on $V(\mathbf{F}_q)$.  It is known \cite[Lemma 6.7.1]{N} that  $Fr_q^i=q^{\frac{i}{2}}\Theta_q^i$,
where  $\Theta_q^i\in \operatorname{Sp}(b_i;\mathbf{R})$ is a symplectic unitary matrix. Since $q=1^r=1$, we conclude that  
 $Fr_q^i=\Theta_q^i$. 

 \medskip
(ii) Recall that the natural transformation $V(\mathbf{F}_q)\to \mathcal{O}_{\operatorname{Mk}_q}$ maps 
$\Theta_q^i$ to a positive matrix $M_q^i=(A, I, I, 0)$ \cite[Lemma 6.7.3 and Remark 6.7.2]{N}.  

 \medskip
(iii) On the other hand, $\operatorname{Mk}_q^i=q^{\frac{i}{2}}M_q^i$ \cite[Definition 6.7.1]{N}. 
In view of $q=1^r=1$, one gets $\operatorname{Mk}_q^i=M_q^i$. 

 \medskip
(iv) It was shown in item (v) of the proof of Lemma \ref{lm3.1},  that  $\operatorname{Mk}_{1^r}^i\cong A^r_i$,
where we apply the same argument to each $0\le i\le 2n$. 

 \medskip
(v) We conclude from items (i)-(iv), that  $Fr_{1^r}^i$ is acting on the cohomology $H^i(V; \mathbf{Q}_{\ell})$ 
by the linear transformations  $A^r_i$ for all $0\le i\le 2n$. 

\bigskip
Lemma \ref{lm3.2} is proved.
\end{proof}

\begin{lemma}\label{lm3.3}
$|V(\mathbf{F}_{1^r})|=\sum_{i=0}^{2n} (-1)^i\operatorname{tr} (A_i^r)$.
\end{lemma} 
\begin{proof}
Roughly speaking, Lemma \ref{lm3.3} follows from Theorem \ref{thm2.1} and Lemma \ref{lm3.2}.  
We shall proceed in the following steps.

\medskip
(i) Notice that $A_i$ are the matrix representations of the algebraic units $\varepsilon_i\in K_i$ in the standard basis 
of the integers in the fields $K_i$.  In particular, $\operatorname{tr} (\varepsilon_i^m)=\operatorname{tr} (A_i^m)$
for any power $m\ge 1$.

\medskip
(ii) On the other hand, repeating the argument of  item (v) of the proof of Lemma \ref{lm3.1}, 
one concludes that $\{\pi_i(q)=r ~|~0\le i\le 2n\}$  whenever $q=1^r$. 

\medskip
(iii)  To finish the proof of Lemma \ref{lm3.3}, it remains to apply Theorem \ref{thm2.1} with   $q=1^r, ~\pi_i(q)=r$ and 
 $\operatorname{tr} (\varepsilon_i^{\pi_i(q)})=\operatorname{tr} (A_i^{\pi_i(q})$.

\bigskip
Lemma \ref{lm3.3} is proved. 
\end{proof}

\bigskip
Theorem \ref{thm1.1} follows from Lemmas \ref{lm3.1}-\ref{lm3.3}.

\subsection{Proof of Corollary \ref{cor1.2}}
Roughly speaking, Corollary \ref{cor1.2} follows from \cite[Theorem 2]{Nik3};
the latter proves Weil's Conjectures (except for an analogue of the Riemann Hypothesis)
in terms of the trace cohomology (Section 2.3).   
We shall split the proof into the series of lemmas.

\begin{lemma}\label{lm3.4}
 The zeta function $\zeta_{V(\mathbf{F}_1)}(s)=\frac{P_1(s)\dots P_{2n-1}(s)}{P_0(s)\dots P_{2n}(s)}$,
where  $P_i(s)\in \mathbf{Z}[s]$ and $P_0(s)=P_{2n}(s)=1-s$. 
\end{lemma}
\begin{proof}
(i) Recall \cite[Theorem 2 (i)]{Nik3} that axioms of the trace cohomology imply: 
\begin{equation}\label{eq3.1}
\zeta_{V(\mathbf{F}_q)}(s)=\frac{P_1(s)\dots P_{2n-1}(s)}{(1-s) P_2(s) \dots P_{2n-2}(s)(1-q^ns)},
\end{equation}
where $P_i(s)\in \mathbf{Z}[s]$ and $q=p^r$.

\medskip
(ii) We shall repeat the argument \cite[Lemma 5]{Nik3} for the value $q=1^r=1$.
It was shown in Section 3.1, that $\omega_i^r$ in \cite[Lemma 5]{Nik3} can be writen
in a matrix form $A_i^r$, such that $\operatorname{tr}   (\omega_i^r)=\operatorname{tr} (A_i^r)$
(Lemma \ref{lm3.3}). 

\medskip
(iii) One can apply \cite[formula (19)]{Nik3} to conclude that $\zeta_{V(\mathbf{F}_1)}(s)$ is a rational
function (\ref{eq3.1}) with $P_0(s)=P_{2n}(s)=1-s$. 

\bigskip
Lemma \ref{lm3.4} is proved.
\end{proof}

\begin{lemma}\label{lm3.5}
The zeta function
$\zeta_{V(\mathbf{F}_1)}(s)$ satisfies the functional equation  $\zeta_{V(\mathbf{F}_1)}\left(\frac{1}{s} \right)=\pm s^{\chi(V)} \zeta_{V(\mathbf{F}_1)}(s)$,
where $\chi(V)$ is the Euler-Poincar\'e characteristic of $V(\mathbf{C})$.
\end{lemma}
\begin{proof}
(i) Recall \cite[Theorem 2 (ii)]{Nik3} that using the trace cohomology one gets: 
\begin{equation}\label{eq3.2}
\zeta_{V(\mathbf{F}_q)}
\left(\frac{1}{q^ns}\right)=\pm q^{n\frac{\chi(V)}{2}} s^{\chi(V)} \zeta_{V(\mathbf{F}_q)}(s). 
\end{equation}

\medskip
(ii)  Let  $q=1^r=1$.  Observe  that the argument of item (ii) in the proof of \cite[Lemma 5]{Nik3}
is valid for the value $q=1$. 

\medskip
(iii) Thus one  obtains  from (\ref{eq3.2}) that  
 $\zeta_{V(\mathbf{F}_1)}\left(\frac{1}{s}\right)=\pm s^{\chi(V)} \zeta_{V(\mathbf{F}_1)}(s)$.

\bigskip
Lemma \ref{lm3.5} is proved. 
\end{proof}

\begin{lemma}\label{lm3.6}
$\deg~P_i(s)=b_i$, where $b_i$ is the $i$-th Betti number of $V(\mathbf{C})$. 
\end{lemma}
\begin{proof}
This is a straigtforward fact following from \cite[formula (19)]{Nik3}
which says that $\operatorname{tr}~(\omega_i^r)=\lambda_1^r+\dots+\lambda_{b_i}^r$,
where $\lambda_j$ are roots of $P_i(s)$ and   $\operatorname{tr}~(\omega_i^r)=\operatorname{tr}~(A_i^r)$ (Section 3.1).
Lemma \ref{lm3.6} is proved.  
\end{proof}

\subsection{Proof of Corollary \ref{cor1.3}}
Roughly speaking, Corollary \ref{cor1.3}  follows from the  inclusion  $\mathbb{A}_V \subset\mathbb{A}_V\rtimes_{\sigma_A}\mathbf{Z}$
and the well-known isomorphism $\mathbb{A}_V\rtimes_{\sigma_A}\mathbf{Z}\cong \mathcal{O}_A\otimes\mathscr{K}$
  \cite[Exercise 10.11.9 (b)]{B}.  Such a  data imply a regular map (injective homomorphism) $\mathbb{A}_V\to \mathcal{O}_A\otimes\mathscr{K}$
  which gives rise to the morphism of schemes  $\operatorname{Spec} ~(\mathbf{Z})\to \operatorname{Spec} ~(\mathbf{F}_{1})\simeq \{\operatorname{pt}\}$. 
  We pass to a detailed argument.

\begin{lemma}\label{lm3.7}
$\operatorname{Spec}\mathbb{A}_V\simeq V(k)$.
\end{lemma}
\begin{proof}
(i) Let $C(\operatorname{Spec}\mathbb{A}_V)$ be commutative $C^*$-algebra of continuous complex-valued 
functions on the locally compact space $\operatorname{Spec}\mathbb{A}_V$. Consider an automorphism $\alpha$
of $C(\operatorname{Spec}\mathbb{A}_V)$, such that $C(\operatorname{Spec}\mathbb{A}_V)\rtimes_{\alpha}\mathbf{Z}\cong \mathbb{A}_V$. 

\medskip
(ii) On the other hand, it is known that $\mathbb{A}_V\cong C(V)\rtimes\mathbf{Z}$, where the crossed product is taken by a $\ast$-coherent 
automorphism of $V(k)$ \cite[Lemma 5.3.2]{N}. 

\medskip
(iii) The Takai duality \cite[Section 8.6.7]{F} for the crossed products in items (i) and (ii)   implies  an isomorphism 
of the commutative $C^*$-algebras $C(\operatorname{Spec}\mathbb{A}_V)\cong C(V)$. 
The Gelfand Theorem says that the topological spaces $\operatorname{Spec}\mathbb{A}_V$ and $V(k)$ must be homeomorphic,
i.e. $\operatorname{Spec}\mathbb{A}_V\simeq V(k)$.

\bigskip
Lemma \ref{lm3.7} is proved. 
\end{proof}

\begin{lemma}\label{lm3.8}
$\operatorname{Spec} ~(\mathcal{O}_A\otimes\mathscr{K}) \simeq V(\mathbf{F}_1)$.
\end{lemma}
\begin{proof}
Lemma \ref{lm3.8} follows from Theorem \ref{thm1.1} when $r=1$. Indeed, 
the covariant functor $V(\mathbf{F}_1)\to \mathcal{O}_A$ extends to the tensor product
 $\mathcal{O}_A\otimes\mathscr{K}$.  Let $\alpha$ be an isomorphism of the 
 of the $C^*$-algebras   $\mathcal{O}_A\otimes\mathscr{K}$ and we let $h_{\alpha}$ be the
 corresponding map on the varieties $V(\mathbf{F}_1)$.  Since $\alpha$ also defines a homeomorphism 
 of the locally compact space $\operatorname{Spec} ~(\mathcal{O}_A\otimes\mathscr{K})$, we can introduce
 a minimal topology on  $V(\mathbf{F}_1)$, such that $h_{\alpha}$ is continuous in this topology. 
 Clearly, one gets a homeomorphism between the corresponding topological spaces, i.e.   
 $\operatorname{Spec} ~(\mathcal{O}_A\otimes\mathscr{K}) \simeq V(\mathbf{F}_1)$.
 Lemma \ref{lm3.8} is proved.

\end{proof}

\begin{lemma}\label{lm3.9}
The map $\operatorname{Spec} ~(\mathbf{Z})\to \operatorname{Spec} ~(\mathbf{F}_{1})\simeq \{\operatorname{pt}\}$
is a morphism of schemes. 
\end{lemma}
\begin{proof}
(i) Let us review the construction of the map $\operatorname{Spec} ~(\mathbf{Z})\to \operatorname{Spec} ~(\mathbf{F}_{1})\simeq \{\operatorname{pt}\}$.
The embedding of the $C^*$-algebras $\mathbb{A}_V \subset\mathbb{A}_V\rtimes_{\sigma_A}\mathbf{Z}\cong \mathcal{O}_A\otimes\mathscr{K}$
  \cite[Exercise 10.11.9 (b)]{B} gives rise to an injective homomorphism: 
  \begin{equation}\label{eq3.3}
  \mathbb{A}_V\to \mathcal{O}_A\otimes\mathscr{K}. 
\end{equation}

\medskip
(ii) In view of Lemmas \ref{lm3.8} and \ref{lm3.9}, one gets from (\ref{eq3.3}) a continuous map:
  \begin{equation}\label{eq3.4}
  V(k) \simeq \operatorname{Spec}\mathbb{A}_V
  \to 
 \operatorname{Spec} ~(\mathcal{O}_A\otimes\mathscr{K}) \simeq V(\mathbf{F}_1).
\end{equation}
The map (\ref{eq3.4}) can be viewed as a reduction of projective variety $V(k)$ modulo a prime ideal 
$\mathscr{P}\subset k$
over  $p=1$. 

\medskip
(iii) Denote by $O_k$ the ring of integers of the number field $k$. 
Since $V(k)$  and $V(\mathbf{F}_1)$ are arithmetic schemes,
the map (\ref{eq3.4}) induces a morphism 
$\operatorname{Spec} ~(O_k)\to\operatorname{Spec} ~(\mathbf{F}_1)\simeq \{\operatorname{pt}\}$. 
In particular, if $k\cong\mathbf{Q}$,  then one gets a morphism of the schemes
$\operatorname{Spec} ~(\mathbf{Z})\to \operatorname{Spec} ~(\mathbf{F}_{1})\simeq \{\operatorname{pt}\}$.

\bigskip
Lemma \ref{lm3.9} is proved.
\end{proof}

\bigskip
Corollary \ref{cor1.3} follows from Lemma \ref{lm3.9}.

\section{Complex multiplication}
Let $\mathscr{E}_{CM}$ be an elliptic curve with complex multiplication written  in the Legendre form
$y^2=x(x-1)(x-\lambda_{CM})$. Roughly speaking, the following result says that 
commutative rings cannot be used to study geometry of $\mathbf{F}_1$. 
\begin{theorem}\label{thm4.1} 
Every map of the form:
\begin{equation}\label{eq4.1}
\mathscr{E}_{CM}(\mathbf{F}_{1^r})\longrightarrow \mathbf{F}_{1^r}[x,y] ~/ ~y^2=x(x-1)(x-\lambda_{CM})
\end{equation}
is a trivial functor.
\end{theorem}
\begin{proof}
(i) Let $L_{\tau}=\mathbf{Z}+\mathbf{Z}\tau$ be a lattice, such that $\tau\in O_k$, where $O_k$ is the ring of integers of an imaginary quadratic number $k$. 
We identify $\mathscr{E}_{CM}(K):= \mathbf{C}/L_{\tau}$, where $K\cong\mathbf{Q}(j(\mathscr{E}_{CM}))$ is a number field generated by 
the $j$-invariant of $\mathscr{E}_{CM}$. 

\medskip
(ii) Let $q=p^r$ and denote by $\mathscr{P}\subset K$ the prime ideal over a good prime $p$ of the elliptic curve 
$\mathscr{E}_{CM}(K)$.  The reduction of $\mathscr{E}_{CM}(K)$ modulo $\mathscr{P}$ will be denoted by $\mathscr{E}_{CM}(\mathbf{F}_q)$. 
Let $\psi(\mathscr{P})$ be the Gr\"ossencharacter associated to $\mathscr{P}$ and let $[\psi(\mathscr{P})]\in O_k$  be an endomorphism of  
$\mathscr{E}_{CM}$ generated by the Frobenius map $Fr_q$ [Silverman 1994] \cite[Chapter II \S 9]{S}.

\medskip
(iii)  Recall that 
\begin{equation}\label{eq4.2}
\zeta_{\mathscr{E}_{CM}(\mathbf{F}_q)}(s)=\frac{1-as+qs^2}{(1-s)(1-qs)},
\end{equation}
where $a=\operatorname{tr} ~[\psi(\mathscr{P})]$ [Silverman 1994] \cite[p. 175-176]{S}. 

\medskip
(iv) Assume to the contrary, that there exists a non-trivial natural map (a functor) of the form (\ref{eq4.1})
which preserves   formula (\ref{eq4.2}) for the value $q=1^r=1$. 
Since the norm of the algebraic integer  $[\psi(\mathscr{P})]\in O_k$ is equal to $q=1$, we conclude that 
$[\psi(\mathscr{P})]$ is a unit of the imaginary quadratic field $k$.
But it is known that there exists only a finite set  $\{\pm 1, \pm i, \frac{\pm 1\pm i\sqrt{3}}{2}\}$ of  such units as $k$ runs through
all possible imaginary quadratic fields.  In other words, we have  $a\in\{0; \pm 1\}$  in formula (\ref{eq4.2}).   Hence functor (\ref{eq4.1}) is
 trivial for all  but  three  of the elliptic curves $\mathscr{E}_{CM}(\mathbf{F}_{1^r})$. 

\bigskip
This argument finishes the proof of Theorem \ref{thm4.1}. 
\end{proof}

\begin{remark}\label{rmk4.2}
The problem of units underlying Theorem \ref{thm4.1} is resolved by the real multiplication 
[Manin 2004] \cite{Man2}  and the corresponding 
Cuntz-Krieger algebras; hence Theorem \ref{thm1.1}.  
\end{remark}

\bigskip
An interplay between the Riemann hypothesis (RH) and geometry of $\mathbf{F}_1$ has been reported by 
many authors, e.g.  [Kapranov \& Smirnov 1995] \cite{KapSmi1}, [Manin 1995] \cite[Section 1.6]{Man1},
[Connes, Consani \& Marcolli 2009] \cite{ConConMar1}, [Lorscheid 2018]  \cite{Lor1}  and others. 
In view of Corollary \ref{cor1.3},  we conclude by the following 
\begin{conjecture}\label{cnj4.3}
The RH is equivalent to such for the elliptic curve $\mathscr{E}_{CM}(\mathbf{F}_{1^r})$
with complex multiplication by the Eisenstein units $\{\frac{\pm 1\pm i\sqrt{3}}{2}\}$. 
\end{conjecture}

\section*{Data availability}
  
  Data sharing not applicable to this article as no datasets were generated or analyzed during the current study.
   
\section*{Conflict of interest}
On behalf of all co-authors, the corresponding author states that there is no conflict of interest.
  

\section*{Funding declaration}
The author was partly supported by the NSF-CBMS grant 2430454.

\bibliographystyle{amsplain}

\begin{thebibliography}{99}

\bibitem{B}
 B.~Blackadar, \textit{$K$-Theory for Operator Algebras}, MSRI Publications,
 Springer, 1986.


\bibitem{ConConMar1}
A.~Connes, C. ~Consani and M.~Marcolli,
\textit{Fun with $\mathbb{F}_1$},
J.  Number Theory {\bf 129} (2009),  1532-1561. 



\bibitem{CunKri1}
J. ~Cuntz and W. ~Krieger, \textit{A class of $C^*$-algebras and topological Markov
chains}, Invent. Math. {\bf 56} (1980), 251-268. 



\bibitem{Dei1}
A.~Deitmar,
\textit{Remarks on zeta functions and $K$-theory over $\mathbb{F}_1$},
Proc. Japan Acad. {\bf 82}, Ser. A (2006),  141-146. 

\bibitem{F}
P.~A.~Fillmore,  
\textit{A User's Guide to Operator Algebras}, Canadian Mathematical
Society Monographs, J.~Wiley \& Sons, Inc. {\bf 1996}. 




\bibitem{Han2}
D.~Handelman, \textit{Positive matrices and dimension groups affiliated
to $C^*$-algebras and topological Markov chains}, J. Operator
Theory {\bf 6} (1981),  55-74.


\bibitem{KapSmi1}
M.~M.~Kapranov and A.~L.~Smirnov,
\textit{Cohomology determinants and reciprocity
laws: number field case}, {\bf 1995} (unpublished)

\bibitem{Lor1}
O.~Lorscheid, 
\textit{$\mathbb{F}_1$ for everyone}, 
Jahresber. Dtsch. Math.-Ver. {\bf 120} (2018),  83-116.


\bibitem{Man1}
Yu. ~I.~Manin,
\textit{Lectures on zeta functions and motives},
Ast\'erisque {\bf 228} (1995), 121-163. 

\bibitem{Man2}
Yu.~I.~Manin, 
\textit{Real multiplication and noncommutative geometry}, in
``The legacy of Niels Hendrik Abel'', 685-727, Springer, Berlin, 2004.



\bibitem{Nik1}
I.~V.~Nikolaev,
\textit{On traces of Frobenius endomorphisms}, 
Finite Fields Appl. {\bf 25} (2014), 270-279. 

\bibitem{Nik3}
I.~V.~Nikolaev,
\textit{Remark on Weil's conjectures},
Adv. Pure Appl. Math. {\bf 7} (2016),  213-221. 

\bibitem{N}
I. ~V. ~Nikolaev, \textit{Noncommutative Geometry},
De Gruyter Studies in Math. {\bf 66}, Second Edition, Berlin, 2022.



\bibitem{Nik2}
I.~V.~Nikolaev, 
\textit{Quantum arithmetic of Drinfeld modules},
Constr. Math. Anal. {\bf 9} (2026),  39-46. 


\bibitem{S}
J.~H.~Silverman,  \textit{Advanced Topics in the Arithmetic of Elliptic Curves},
GTM  {\bf 151}, Springer 1994.


\bibitem{Sou1}
Ch. ~Soul\'e, 
\textit{Varieties over field with one element}, Mosc. Math. J.
{\bf 4} (2004), 217-244.


\bibitem{Tit1}
J.~Tits, 
\textit {Sur les analogues alg\'ebriques des groupes semi-simples
complexes}, 
Colloque d'alg\`ebre sup\'erieure, tenu \`a Bruxelles du 19 au 22 décembre 1956,
Centre Belge de Recherches Math\'ematiques, Louvain, Paris:
Librairie Gauthier-Villars {\bf 1957}, pp. 261-289. 


\bibitem{ToeVaq1}
B.~To\"en and M.~Vaqui\'e, 
\textit{Au-dessous de $\operatorname{Spec} \mathbb{Z}$,}
Journal of K-Theory {\bf 3}  (2009),  437-500.


\end{thebibliography}


\end{document}